\documentclass[12pt]{amsart}
\usepackage{latexsym,amssymb,amsfonts,amsmath,graphicx,fullpage,url}
\usepackage{enumerate}
\usepackage{color}
\usepackage{cite}
\usepackage[all]{xy}
\usepackage{verbatim}
\usepackage{hyperref}
\newtheorem{theorem}{Theorem}[section]
\newtheorem*{mainthmq1}{Theorem~\ref{thm:t=1}}
\newtheorem*{mainthmq0}{Theorem~\ref{thm:t=0}}
\newtheorem*{mainthmqinv}{Theorem~\ref{thm:q-1}}
\newtheorem*{corqinv}{Corollary~\ref{cor:q-1}}
\newtheorem*{mainthmq1a}{Theorem~\ref{thm:q1a}}
\newtheorem{proposition}[theorem]{Proposition}
\newtheorem{lemma}[theorem]{Lemma}
\newtheorem{corollary}[theorem]{Corollary}
\theoremstyle{definition}
\newtheorem{definition}[theorem]{Definition}
\newtheorem{example}[theorem]{Example}
\newtheorem{conjecture}[theorem]{Conjecture}

\theoremstyle{remark}
\newtheorem{remark}[theorem]{Remark}
\numberwithin{equation}{section}
\newcommand{\stat}{\operatorname{stat}}
\def\F{\mathcal{F}}

\def\F{{\mathcal{F}}}

\def\Sn{{\mathfrak{S}}_n}

\def\a{{\bf a }}

\newcommand{\ZZ}{\mathbf Z}

\newcommand{\arxiv}[1]{\url{http://arxiv.org/abs/#1}}

\newcommand{\be}{\begin{equation}}
\newcommand{\ee}{\end{equation}}

\newcommand{\bd}{\begin{definition}}
\newcommand{\ed}{\end{definition}}
\newcommand{\bt}{\begin{theorem}}
\newcommand{\et}{\end{theorem}}
\newcommand{\bl}{\begin{lemma}}
\newcommand{\el}{\end{lemma}}
\newcommand{\bp}{\begin{proposition}}
\newcommand{\ep}{\end{proposition}}
\newcommand{\bc}{\begin{corollary}}
\newcommand{\ec}{\end{corollary}}

\def\i{{\it in}}

\DeclareMathOperator{\Ehr}{Ehr}
\DeclareMathOperator{\hilb}{Hilb}
\DeclareMathOperator{\inv}{inv}
\DeclareMathOperator{\Inv}{Inv}
\DeclareMathOperator{\dinv}{dinv}
\DeclareMathOperator{\area}{area}
\DeclareMathOperator{\codeg}{codeg}
\DeclareMathOperator{\pmaj}{pmaj}

\renewcommand{\b}{\mathbf}
\newcommand{\RR}{\mathbf R}
\renewcommand{\subset}{\subseteq}
\newcommand{\od}{\bar d}
\renewcommand{\mathbb}{\mathbf}

\newtheorem*{theorem1*}{Theorem \ref{thm:main}}

\newcommand{\old}[1]{}

\title{Flow polytopes and the space of diagonal harmonics}

\author{Ricky Ini Liu}
\address{Department of Mathematics, North Carolina State University, Raleigh, NC}
\email{riliu@ncsu.edu}

\author{Karola M\'esz\'aros}
\address{Department of Mathematics, Cornell University, Ithaca, NY}
\email{karola@math.cornell.edu}

\author{Alejandro H. Morales}
\address{Department of Mathematics, University of California, Los Angeles, Los Angeles, CA}
\email{ahmorales@math.ucla.edu}

\thanks{M\'esz\'aros was partially supported by a National Science Foundation Grant (DMS 1501059). Morales was partially supported by an AMS-Simons travel grant.}

\begin{document}

\begin{abstract}
	A result of Haglund implies that the $(q,t)$-bigraded Hilbert series of the space of diagonal harmonics is a $(q,t)$-Ehrhart function of the flow polytope of a complete graph with netflow vector $(-n, 1, \dots, 1).$ We study the $(q,t)$-Ehrhart functions of flow polytopes of threshold graphs with arbitrary netflow vectors. Our results generalize previously known specializations of the mentioned bigraded Hilbert series at $t=1$, $0$, and $q^{-1}$. As a corollary to our results, we obtain a proof of a conjecture of Armstrong, Garsia, Haglund, Rhoades and Sagan about the $(q, q^{-1})$-Ehrhart function of the flow polytope of a complete graph with an arbitrary netflow vector.
\end{abstract}

\maketitle

\section{Introduction}
\label{sec:intro} 

	The \textit{space of diagonal harmonics}
	\[DH_n=\left\{  f \in \mathbf{C}[x_1, \ldots, x_n, y_1, \ldots, y_n] \quad \left| \quad\sum_{i=1}^n \frac{\partial^h}{\partial x_i^h} \frac{\partial^k}{\partial y_i^k} f=0 \; \text{ for all } \; h+k>0 \right. \right\}\]
	was introduced by Garsia and Haiman \cite{gh1} in their study of Macdonald polynomials. Haiman \cite{hai} proved using algebro-geometric arguments that it has dimension $(n+1)^{n-1}$ as a vector space over $\mathbf{C}$. The space $DH_n$ is naturally bigraded by the degree of the variables $x_i$ and $y_j$. Thus, one can obtain a $q,t$-analogue of $(n+1)^{n-1}$ by considering the bigraded Hilbert series of $DH_n$, which we denote by $\hilb_{q,t}(DH_n)$. This is a symmetric polynomial in $q$ and $t$ with nonnegative coefficients.
	
	The number $(n+1)^{n-1}$ counts spanning trees of the complete graph on $n+1$ vertices or parking functions of size $n$. A combinatorial model for this bigraded Hilbert series in terms of these objects was conjectured by Haglund and Loehr \cite{hl} in 2002 and settled in 2015 by Carlsson and Mellit \cite{CM} in their proof of the more general Shuffle Conjecture \cite{HHLRU}. Stated in terms of parking functions, the result is the following (see \cite[Conj. 2]{hl}).
	
	\begin{theorem}[Carlsson--Mellit \cite{CM}, Hilbert series conjecture of Haglund--Loehr \cite{hl}] \label{HLhilb}
		\begin{equation}
			\operatorname {Hilb}_{q,t}(DH_n) = \sum_{p \in \mathcal{P}_n} q^{\area(p)}t^{\dinv(p)},
		\end{equation}
		where $\mathcal{P}_n$ denotes parking functions of size $n$.
	\end{theorem}
	
	For more background on $DH_n$ and the Shuffle Conjecture see \cite{Bergeron,Haglund1,Haglund2}. For the definition of $\area$ and $\dinv$ on parking functions see \cite[\S 1, \S 2]{hl}. Special cases of this Hilbert series when $t=1,0,q^{-1}$ are combinatorially appealing:
	\begin{align}
		\hilb_{q,1}(DH_n) &= \sum_{p \in \mathcal{P}_n} q^{\area(p)}, \label{case:q1}\\
		\hilb_{q,0}(DH_n) &= \sum_{w \in S_n} q^{\inv(w)}=[n]_q!, \label{case:q0}\\
		q^{\binom{n}{2}} \hilb_{q,q^{-1}}(DH_n) &=  [n+1]_q^{n-1}, \label{case:qqinv}
	\end{align}
	where $S_n$ is the symmetric group of size $n$, $\inv(w)$ is the number of inversions of the permutation $w$, and $[k]_q = 1+q+\cdots + q^{k-1}$. The right hand side of \eqref{HLhilb} evaluated at $(q,1)$ is trivially the right hand side of \eqref{case:q1}.  The fact that the right hand side of \eqref{HLhilb} evaluated at $(q,0)$ yields the right hand side of \eqref{case:q0} follows from \cite[Theorem 5.3]{Haglund1} together with the fact that the major index and number of inversions are equidistributed over $S_n$ \cite[\S 1.4]{EC}. Finally, Loehr \cite{L} showed the case $(q,q^{-1})$ combinatorially. Showing directly that the Hilbert series has these evaluations is highly nontrivial and is due to Haiman \cite{hai94}.
	
	Before the proof of Haglund and Loehr's Hilbert series conjecture in \cite{CM}, Haglund \cite{Hag} gave an expression for the Hilbert series as a  weighted sum over certain upper triangular matrices called {\em Tesler matrices} \cite[\href{http://oeis.org/A008608}{A008608}]{OEIS}. In \cite{tesler}, M\'esz\'aros, Morales, and Rhoades noticed that these matrices can be easily reinterpreted as integer flows on the complete graph $K_{n+1}$ with netflow vector $(-n, 1, \ldots,1)$. With this interpretation, Haglund's result states that the Hilbert series equals a weighted sum over the lattice points of the polytope of flows on $K_{n+1}$ with netflow vector $(-n, 1, \ldots,1)$. Denoting this sum by $\Ehr_{q,t}(\F_{K_{n+1}}(-n, 1, \ldots, 1))$ (see Section \ref{sec:preliminaries} for the precise definition of flows and their weight), Haglund's result can be restated as follows.
	
	\begin{theorem}[Haglund \cite{Hag}] \label{thm:1}
		\[\operatorname {Hilb}_{q,t}(DH_n)= \Ehr_{q,t}(\F_{K_{n+1}}(-n, 1, \ldots, 1)).\]
	\end{theorem}
	
	Combining Theorems~\ref{HLhilb} and \ref{thm:1}, we obtain intriguing combinatorial identities between $\Ehr_{q,t}(\F_{K_{n+1}}(-n,1,\ldots,1))$ and $(q,t)$-analogues (for $t$, $t=1$, $t=0$, and $t=q^{-1}$) of the number of parking functions of size $n$:
	\begin{equation} \label{eq:motivation1}
		\Ehr_{q,t}(\F_{K_{n+1}}(-n,1,\ldots,1)) =  \sum_{p \in \mathcal{P}_n} q^{\area(p)}t^{\dinv(p)}.
	\end{equation}
	 
	There are many natural bijections between spanning trees of $K_{n+1}$ and parking functions of size $n$. Correspondingly, there are various statistics $(\stat_1, \stat_2)$ on trees that can be used to rewrite the right hand side of \eqref{eq:motivation1} as a sum over spanning trees of $K_{n+1}$:
	
	\begin{equation} \label{eq:motivation}
		\Ehr_{q,t}(\F_{K_{n+1}}(-n,1,\ldots,1)) =  \sum_T q^{\stat_1(T)}t^{\stat_2(T)},
	\end{equation}
	where the sum is over all spanning trees $T$ of $K_{n+1}$: see for instance \cite[\S 4]{hl} for one example. Equations \eqref{case:q1}, \eqref{case:q0}, and \eqref{case:qqinv} can be rewritten as 
	\begin{align}
		\Ehr_{q,1}(\F_{K_{n+1}}(-n,1,\ldots,1)) &= \sum_{T} q^{\inv(T)}, \label{case:q11}\\
		\Ehr_{q,0}(\F_{K_{n+1}}(-n,1,\ldots,1)) &= [n]_q!, \label{case:q01}\\
		q^{\binom{n}{2}} \Ehr_{q,q^{-1}}(\F_{K_{n+1}}(-n,1,\ldots,1))&=  [n+1]_q^{n-1}, \label{case:qqinv1}
	\end{align}
	where on the right hand side of \eqref{case:q11}, $T$ ranges over all spanning trees of $K_{n+1}$, and $\inv(T)$ is the number of inversions of $T$ (see Section \ref{sec:q1} for the definition of $\inv$ statistic and the correspondence to the $\area$ statistic on parking functions). It is then natural to verify these identities directly. Doing so in the general case $(q,t)$ would give an alternative proof of the now settled Haglund--Loehr conjecture. Progress in this direction started with Levande \cite{PL} who verified the cases $(q,0)$ using a sign-reversing involution. Armstrong et al.\ \cite{AGHRS} verified the case $(q,1)$. We verify directly the $(q,q^{-1})$ case in this paper.
	
	More generally, one could extend the identity \eqref{eq:motivation} to flows with other netflow vectors (in \cite{AGHRS}, these are called \emph{generalized Tesler matrices}) or to other graphs. The former was done in \cite{AW} for the $(q,0)$ case for binary netflows on $K_{n+1}$ extending the involution approach of Levande. Formulas for the $(q,1)$ case for positive integral netflows were given in \cite{AGHRS} and for integral flows in \cite{AW}. We generalize the known formulas for $\Ehr_{q,t}(\F_{K_{n+1}}(-n, 1, \ldots, 1))$ for $t=1,0,q^{-1}$ (as in equations \eqref{case:q11}, \eqref{case:q01}, and \eqref{case:qqinv1}) to a family of graphs called \emph{threshold graphs} \cite{TG} with arbitrary positive integral netflows. There are $2^n$ such graphs with $n+1$ vertices including the complete graph.
	
	We now summarize our main results. First we state the case $t=1$ for netflow $(-n,1,\ldots,1)$, which implies \eqref{case:q11} when $G$ is the complete graph.

	\begin{mainthmq1}
		Let $G$ be a threshold graph. Then
		\begin{equation*}
			\Ehr_{q,1}(\F_G(-n,1,\ldots,1)) = t_G(1,q) = \sum_T q^{\inv(T)},
		\end{equation*}
		where $t_G$ is the Tutte polynomial of $G$, and $T$ ranges over all spanning trees $T$ of $G$. 
	\end{mainthmq1}
	
	This relationship between $\Ehr_{q,1}(\F_G(-n, 1, \dots, 1))$ and the Tutte polynomial of $G$ extends to general positive flows as follows.
	
	\begin{mainthmq1a}
		For a connected threshold graph $G$ and $\b a \in \ZZ^n_{>0}$, let $\widetilde G$ be the multigraph obtained from $G$ by replacing each edge $(i,j)$ with $a_{\max\{i,j\}}$ parallel edges. Then \[\Ehr_{q,1}(\F_G(-\textstyle\sum_{i} a_i,a_1,\ldots,a_n)) = t_{\widetilde G}(1,q),\] where $t_{\widetilde G}$ is the Tutte polynomial of $\widetilde G$.
	\end{mainthmq1a}

	We also state the case $t=0$, which implies \eqref{case:q01} when $G$ is the complete graph and $\a = (1,\ldots,1)$.
	
	\begin{mainthmq0}
		Let $G$ be a threshold graph with degree sequence $(d_0,d_1,\ldots,d_n)$ and $\b a \in \ZZ^n_{>0}$. Then
		\begin{equation*}
			\Ehr_{q,0}(\F_G(-\textstyle\sum_{i} a_i,a_1,\ldots,a_n)) = \displaystyle\prod_{i=1}^n q^{\bar d_i(a_i-1)} [\bar{d}_i]_q,
		\end{equation*}
		where $\bar{d}_i = \min\{d_i,i\}$ is the number of vertices $j<i$ adjacent to $i$.
	\end{mainthmq0}

	Lastly, we state the case $t=q^{-1}$, which implies \eqref{case:qqinv1} when $G$ is the complete graph and $\a = (1,\ldots,1)$.

	\begin{mainthmqinv}
		Let $G$ be a threshold graph with degree sequence $(d_0,d_1,\ldots,d_n)$ and $\b a \in \ZZ^n_{>0}$. Then
		\[\operatorname {Ehr}_{q,q^{-1}}(\F_G(-\textstyle\sum_{i} a_i, a_1, \dots, a_n)) = q^{-F}\displaystyle\prod_{i=1}^n b_i(q),\]
	 	where $F = \sum_{i=1}^n \min\{d_i,i\} \cdot a_i -n$ and 	
	 	\[
		 	b_i(q) = \begin{cases}
		 		[(i+1)a_i+\sum_{j=i+1}^{d_i}a_j]_q&\text{if }d_i>i,\\
		 		[a_i]_{q^{i+1}}&\text{if }d_i = i,\\
		 		[a_i]_{q^{d_i+1}}[d_i]_q&\text{if }d_i  < i.
		 	\end{cases}
	 	\]
	\end{mainthmqinv}
	
	As a corollary, we prove a conjecture of Armstrong et al.\ \cite{AGHRS} about the $(q, q^{-1})$-Ehrhart function of the flow polytope of a complete graph with an arbitrary netflow vector. The case $a_1=\cdots =a_n=1$ gives \eqref{case:qqinv1}.
	
	\begin{corqinv}
		For positive integers $a_1,\ldots,a_n$ we have that 
		\[
			\Ehr_{q,q^{-1}}(\F_{K_{n+1}}(-\textstyle\sum_{i} a_i,a_1,\ldots,a_n)) =
			q^{n-\sum_{i=1}^n ia_i} \displaystyle\prod_{i=1}^{n-1} [ (i+1)a_i +
			a_{i+1}+a_{i+2}+\cdots + a_n ]_q.
		\] 
	\end{corqinv}
	
	Our proofs are self-contained and inductive on the netflow of the flow polytope without using machinery from symmetric functions. In the case $(q,0)$ we do not use involutions like Levande in \cite{PL} and Wilson in \cite{AW}. 
	
	The outline of this paper is as follows. In Section \ref{sec:preliminaries} we give the definitions of flow polytopes, $(q,t)$-Ehrhart functions, and threshold graphs. In Section \ref{sec:q1} we calculate $\Ehr_{q,1}(\cdot)$ for flow polytopes of threshold graphs, while in Section \ref{sec:q0} we do the same for the evaluation $(q,0)$. In Section \ref{sec:qq^{-1}} we calculate the evaluation $(q,q^{-1})$ thereby also proving Conjecture 7.1 of Armstrong et al.\ in \cite{AGHRS}. We conclude in Section \ref{sec:conj} with positivity conjectures regarding the general $(q,t)$ case of flow polytopes of threshold graphs.
 
\section{Preliminaries}
\label{sec:preliminaries}

	In this section, we give some background and preliminary results about flow polytopes and threshold graphs.

	\subsection{Flow polytopes and their $(q,t)$-Ehrhart functions}
	
	We first discuss flow polytopes and define the $(q,t)$-Ehrhart functions.
	\begin{definition}
		Let $G=(V,E)$ be an acyclic directed graph on $V = \{0, 1, \dots, n\}$, and let $\b a \in \ZZ^{n+1}$. Let $A_G$ be the $n \times |E|$ matrix with columns $e_i-e_j$ for each directed edge $(i,j)$. Then the \emph{flow polytope} $\F_G(\b a) \subset \RR_{\geq 0}^{E}$ is defined to be
		\[\F_G(\b a) = \{x \in \RR_{\geq 0}^{E} \mid A_G \cdot x = \b a\}.\]
	\end{definition}
	In other words, the flow polytope is the set of all nonnegative flows that can be placed on the edges of $G$ such that the net flow at each vertex is given by $\b a$. By convention, we  orient the edges from $i$ to $j$ if $i>j$. Since the sum of the entries of $\b a$ must be 0 for the flow polytope to be nonempty, we will abuse notation and write, for $\b a = (a_1, \dots, a_n) \in \ZZ^n$, $\F_G(\b a) = \F_G(-\sum a_i, a_1, \dots, a_n)$. We also abbreviate $\F_G = \F_G(-n, 1, 1, \dots, 1)$.
	
\begin{remark} \label{rem:flow2Tesler}
		A \emph{Tesler matrix} is an $n\times n$ upper triangular matrix $B=(b_{i,j})_{1\leq i\leq j \leq n}$ with nonnegative integer entries satisfying for $k=1,\ldots,n$,
		\[
			b_{k,k} + b_{k,k+1} + \cdots + b_{k,n} - (b_{1,k}+b_{2,k}+\cdots + b_{k-1,k})=1.
		\]
		These matrices first appeared in Haglund's study of $DH_n$ \cite{Hag}. By an observation in \cite{tesler}, these matrices are in correspondence with integral flows on $K_{n+1}$ with netflow $(-n,1,1,\ldots,1)$. With the conventions on $\F_G$ in this paper, the correspondence is as follows: an integral flow $A=(a_{ij})_{0\leq j <i \leq n}$  in $\F_{K_{n+1}}$ corresponds to the Tesler matrix $B=(b_{ij})_{1\leq i \leq j \leq n}$ where
		\[
			b_{ij} = \begin{cases}
				a_{n+1-j,0} & \text{ if } i=j,\\
				a_{n+1-i,n+1-j}  & \text{ if } i<j.
			\end{cases}
		\]
		For example, for $n=4$ the correspondence is the following:
		\begin{center}
			\includegraphics{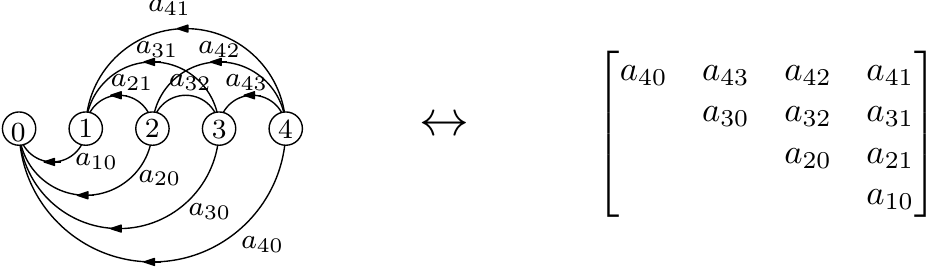}
		\end{center}
		This correspondence can be extended to integral flows on subgraphs $G$ of $K_{n+1}$ by setting the entries corresponding to missing edges of $G$ to zero.
	\end{remark}

	For any nonnegative integer $b$, define the \emph{$(q,t)$-weight}
	\begin{equation} \label{eq:wt}
		wt_{q,t}(b) = \begin{cases}
			\frac{q^b-t^b}{q-t} &\text{ if } b>0,\\
			1 &\text{ if } b=0.
		\end{cases}
	\end{equation}
	For a lattice point $A = (a_{ij}) \in \RR_{\geq 0}^{E}$ with nonnegative entries, define
	\[
		wt_{q,t}(A) = \left(-(1-t)(1-q)\right)^{\#\{a_{ij} > 0\}-n} \cdot \prod_{i,j} wt_{q,t}(a_{ij}),
	\]
	where $\#\{a_{ij}>0\}$ denotes the number of nonzero entries of $A$. Finally, for an integer polytope $\F_G(\a) \subset \RR_{\geq 0}^{E}$, define the \emph{$(q,t)$-weighted Ehrhart function}
	\[\Ehr_{q,t}(\F_G(\b a)) = \sum_{A \in \F_G(\b a) \cap \ZZ^E} wt_{q,t}(A).\]
	Note that if $\b a \in \ZZ_{>0}^n$, then any $A \in \F_G(\b a)$ will have at least $n$ nonzero entries, so $wt_{q,t}(A)$ and hence $\Ehr_{q,t}(\F_G(\b a))$ will be \textit{polynomials} in $q$ and $t$. Moreover, this polynomial by construction is \textit{symmetric} in $q$ and $t$. There is no guarantee, however, that $\Ehr_{q,t}(\F_G(\b a))$ will have nonnegative coefficients, and indeed it will not for general graphs $G$ as illustrated in the next example. 
	
	\begin{example} \label{ex:negqtGraph}
		If $G=K_5\setminus \{(3,4)\}$, then there are $15$ integer flows on $G$ and one can check that 
		\[
			\Ehr_{q,t}(\F_{G}(-4,1,1,1,1)) = q^3t + 2q^2t^2 + qt^3 - 3q^3 - 5q^2t - 5qt^2 - 3t^3 - 5q^2 - 8qt - 5t^2 - 3q - 3t - 1.
		\]
	\end{example}

	\subsection{Threshold graphs} \label{subsec:tg}
	
	We now define threshold graphs, a class of graphs of importance in computer science and optimization. For more information, see \cite{TG} and \cite[Ex. 5.4]{EC}.
	
	\begin{definition}
		A \emph{threshold graph} $G$ is a graph that can be constructed recursively starting from one vertex and no edges by repeatedly carrying out one of the following two steps:
		\begin{itemize}
			\item add a dominating vertex: a vertex that is connected to every other existing vertex; 
			\item add an isolated vertex: a vertex that is not connected to any other existing vertex. 
		\end{itemize}
		We say that a threshold graph $G$ is labeled by \emph{reverse degree sequence} if its vertices are labeled by $0, \dots, n$ in such a way that $d_i \geq  d_j$ for each pair of vertices $i < j$, where $d_i$ is the degree of vertex $i$.
	\end{definition}

	This family of graphs includes the complete graph and the star graph but excludes paths or cycles of $4$ or more vertices. There are $2^{n-1}$ threshold graphs with $n$ unlabeled vertices. The number $t(n)$ of threshold graphs with vertex set $[n]$ has exponential generating function $e^x(1-x)/(2-e^x)$, and $t(n) \sim n!(1-\log(2))/\log(2)^{n+1}$ (e.g. see \cite[\href{http://oeis.org/A005840}{A005840}]{OEIS}). A threshold graph is uniquely determined up to isomorphism by its degree sequence $d(G)=(d_0, d_1, \dots, d_n)$. By convention, we will assume that all our threshold graphs are labeled by reverse degree sequence and that the edges are directed from $i$ to $j$ if $i>j$.

	Alternatively, a graph $G$ is a threshold graph if there exist real weights $w_i$ for each vertex $i = 0, \dots, n$ and a threshold value $t$ such that $i$ and $j$ are adjacent if and only if $w_i + w_j > t$. If the vertices are labeled such that $w_0 > w_1 > \cdots > w_n$, then $G$ is labeled by reverse degree sequence. Note that if $i$ and $j$ are adjacent in $G$, then so are $i'$ and $j'$ for any $i' \leq i$ and $j' \leq j$ (provided $i' \neq j'$).
	
	\begin{remark} \label{rem:tglabeling}
		A threshold graph with $n+1$ vertices can be encoded by a binary sequence $(\beta_0,\ldots,\beta_{n-1}) \in \{0,1\}^n$ where $\beta_i=1$ or $0$ depending on whether vertex $i$ is a dominating or an isolated vertex with respect to vertices $i+1,\ldots,n$. In this labeling, if $i$ and $j$ are adjacent with $i<j$, then $d_i  \geq d_j$ since all vertices at least $i$ are adjacent to $i$ and all vertices smaller than $i$ are either adjacent to both $j$ and $i$ or to neither. Hence when we relabel the vertices by reverse degree sequence the orientation of the edges is preserved. Thus if $G=G(\beta)$ is a threshold graph with the labeling induced from $\beta$, and $G'$ is the graph relabeled by reverse degree sequence, then
		\[
			\Ehr_{q,t}(\F_{G(\beta)}(\a)) \,=\, \Ehr_{q,t}(\F_{G'}(\a')),
		\]
		where $\a'$ is obtained by permuting $\a$ according to the relabeling of the vertices. In the case when the graph is connected ($\beta_0=1$) and $\a=(-n,1,1,\ldots,1)$, then $\a=\a'$ and the equation above becomes
		\[
			\Ehr_{q,t}(\F_{G(\beta)}) \,=\, \Ehr_{q,t}(\F_{G'}).
		\]
		Using the correspondence between integral flows on graphs and Tesler matrices in Remark~\ref{rem:flow2Tesler}, the $n\times n$ matrices corresponding to the flows on threshold graph $G(\beta)$ have zero entries above the diagonal in column $i+1$ if $\beta_i=0$. 
	\end{remark}

	\begin{example}
		The threshold graph $G(1,0,1,0)$ corresponds to the graph $G'$ with reverse degree sequence $(4,3,2,2,1)$. The map between integral flows on $G(1,0,1,0)$ and $G'$ and Tesler matrices is the following:
		\begin{center}
			\includegraphics{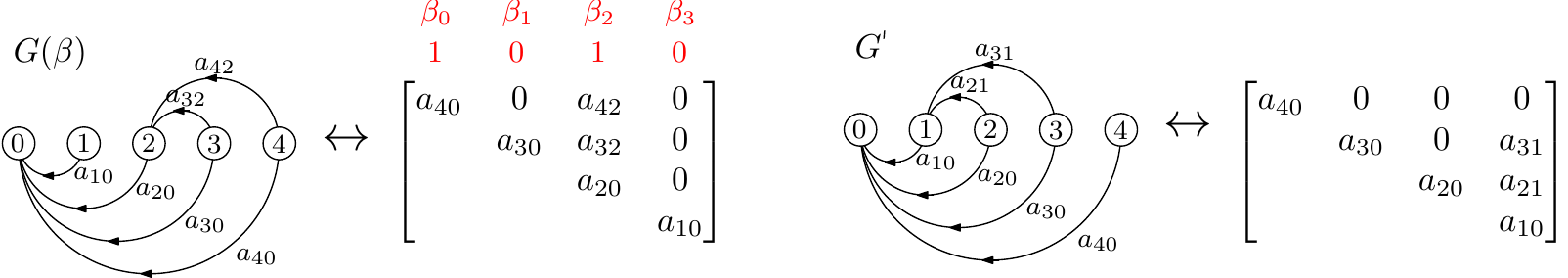}
		\end{center}
	\end{example}

\section{Calculating the $(q,1)$-Ehrhart function}
\label{sec:q1}
	
	In this section, we give a combinatorial formula for the weighted Ehrhart function of the flow polytope $\F_G(\b a)$ when $G$ is a threshold graph and $t=1$. We note that one such proof when $G$ is the complete graph was given by Wilson \cite[\S 6]{AW}. In particular, it will follow that when $q=t=1$, the weighted Ehrhart function evaluates to the number of spanning trees of $G$, or equivalently, to the number of $G$-parking functions.
	
	To begin, we will need some background about spanning trees, inversions, and parking functions, particularly in relation to threshold graphs.
	
	\subsection{Spanning trees and inversions}
	
	One important statistic on spanning trees is the number of inversions. The related notion of $\kappa$-inversions is due to \cite{gessel}. We define both these notions below.
	
	\begin{definition}
		Let $G$ be a graph on $0, 1, \dots, n$, and let $T$ be a spanning tree of $G$ rooted at $r$. We say that $v$ is a \emph{descendant} of $u$ if $u$ lies on the unique path from $r$ to $v$ in $T$. We say $u$ is the \emph{parent} of a vertex $v$ if $v$ is a descendant of $u$ in $T$, and $u$ and $v$ are adjacent in $G$.
	\end{definition}
		
	\begin{definition}
		An \emph{inversion} of $G$ is a pair of vertices $(i,j)$ with $r \neq i > j$ such that $j$ is a descendant of $i$. A \emph{$\kappa$-inversion} of $G$ is an inversion $(i,j)$ such that $j$ is adjacent to the parent of $i$ in $G$. We denote the number of inversions of $T$ by $\inv(T)$ and the number of $\kappa$-inversions by $\kappa(T)$.
	\end{definition}
	
	We will assume our trees are rooted at $r=0$ unless otherwise indicated.
	
	We also briefly recall the definition of the Tutte polynomial of a graph.
	\begin{definition}
		Let $G = (V,E)$ be a multigraph. The \emph{Tutte polynomial} of $G$ is defined by
		\[t_G(x,y) = \sum_{A \subset E} (x-1)^{k(A)-k(E)} (y-1)^{k(A)+|A|-|V|},\]
		where $k(A)$ denotes the number of connected components in the graph $(V,A)$.
	\end{definition}
 
	Define the \emph{inversion enumerator} of $G$ to be
	\[I_G(q) = \sum_{T} q^{\kappa(T)},\]
	where $T$ ranges over all spanning trees of $G$. Gessel shows in \cite{gessel} that $I_G(q)$ has the following properties.
	
	\begin{theorem} \label{thm:gessel} \cite{gessel}
		Let $G$ be a graph on $0, 1, \dots, n$.
		\begin{enumerate}[(a)]
			\item The polynomial $I_G(q)$ does not depend on the labeling of $G$.  In fact, $I_G(q) = t_G(1,q)$.
			\item For any vertex $i \neq 0$, let $\delta_{T, G}(i)$ be the number of descendants of $i$ in $T$ (including $i$ itself) that are adjacent in $G$ to the parent of $i$. Then
			\[I_G(q) = \sum_{T \colon \kappa(T) = 0} \prod_{i=1}^n [\delta_{T,G}(i)]_q,\]
			where the sum ranges over all spanning trees $T$ for which $\kappa(T) = 0$.
		\end{enumerate}
	\end{theorem}
	
	Here we use the standard notation $[k]_q = 1+q+q^2+ \dots + q^{k-1} = \frac{q^k-1}{q-1}$. In the case when $G$ is a threshold graph, these results specialize as follows. Call a spanning tree of $G$ \emph{increasing} if it has no inversions.
	
	\begin{proposition} \label{prop:threshinv}
		Let $G$ be a threshold graph (labeled by reverse degree sequence). Then 
		\[	
			I_G(q) = \sum_T q^{\inv(T)}
			= \sum_{T \text{ increasing}} \prod_{i=1}^n [\delta_{T}(i)]_q,
		\]
		where $\delta_T(i)$ is the number of descendants of $i$ in $T$ (including $i$ itself).
	\end{proposition}
	\begin{proof}
		For a spanning tree of a threshold graph, any inversion $(i,j)$ is a $\kappa$-inversion: $j<i$ implies that any vertex adjacent to $i$ is also adjacent to $j$ in $G$, particularly the parent of $i$ (or see \cite[Proposition 10]{pyy}).
		
		For the second equality, if $j$ is a descendant of $i$, then the parent of $j$ is a descendant of the parent of $i$. Thus in an increasing tree, the parent of $i$ is at most the parent of $j$, so since $j$ is adjacent to the latter, it must also be adjacent to the former in $G$. The result then follows from Theorem~\ref{thm:gessel}.
	\end{proof}

	\subsection{Parking functions}
	
	The following notion of a $G$-parking function due to Postnikov and Shapiro \cite{postnikov-shapiro} generalizes the usual notion of parking function (the latter corresponds to the complete graph). They are also  called \emph{superstable configurations}  in the context of chip-firing.
	
	\begin{definition}
		Let $G = (V,E)$ be a graph on $V = \{0, 1, \dots, n\}$. A \emph{$G$-parking function} is a function $P \colon [n] \to \ZZ_{\geq 0}$ such that, for every nonempty set $S \subset [n]$, there exists $i \in S$ such that $P(i)$ is less than the number of vertices $j \notin S$ adjacent to $i$.
		
		The \emph{degree} of a parking function $P$ is defined to be $\deg P = \sum_{i=1}^n P(i)$. The \emph{codegree} of a parking function $P$ is $\codeg P = g-\deg P$, where $g=|E|-|V|+1$.
	\end{definition}
	
	When $G$ is the complete graph, $P$ is a parking function if and only if, for $k = 1, 2, \dots, n$, there are at least $k$ vertices $i$ such that $P(i) < k$. In the context of ordinary parking functions on the complete graph, the $\codeg$ statistic is usually referred to as $\area$.
	
	In general, $G$-parking functions are in bijection with the spanning trees of $G$. Merino \cite{merino} showed the following relationship (in the context of chip-firing) between parking functions and the Tutte polynomial of $G$.
	
	\begin{theorem}\cite{merino}
		Let $G$ be a graph. Then \[t_G(1, y) = \sum_P y^{\codeg P},\] where the sum ranges over all $G$-parking functions $P$.
	\end{theorem}
	
	In light of Gessel's results on the inversion enumerator of $G$, it follows that the $\kappa$-inversion statistic on spanning trees of $G$ has the same distribution as the codegree statistic on $G$-parking functions. (This was noted in the case of the complete graph by Kreweras \cite{kreweras}.) The authors of \cite{pyy} give an explicit bijection (called the \emph{DFS-burning algorithm}) between spanning trees $T$ and $G$-parking functions $P$ that sends $\kappa(T)$ to $\codeg (P)$. If $G$ is a threshold graph, then this bijection sends $\inv(T)$ to $\codeg(P)$.	

	\subsection{Relation to the Ehrhart function}
	
	We are now ready to state the main result of this section.
	
	\begin{theorem} \label{thm:t=1}
		Let $G$ be a threshold graph. Then
		\[\operatorname {Ehr}_{q, 1}(\F_G) = t_G(1,q) = I_G(q) = \sum_{T} q^{\inv(T)} = \sum_{P} q^{\codeg(P)},\]
		where $T$ ranges over all spanning trees of $G$, and $P$ ranges over all $G$-parking functions.
	\end{theorem}
	\begin{proof}
		Note that for any $A \in \F_G \cap \ZZ^E$, $wt_{q,1}(A) = 0$ unless $A$ has exactly $n$ nonzero entries. Hence to compute $\operatorname {Ehr}_{q,1}(\F_G)$, we need only sum $wt_{q,1}(A)$ over such $A$.
		
		For any $A = (a_{ij}) \in \F_G \cap \ZZ^E$, the set of edges $(i,j)$ for which $a_{ij} \neq 0$ forms a connected subgraph of $G$. Hence if $A$ has exactly $n$ nonzero entries, then these edges must form a spanning tree $T$ of $G$. We claim that such $A$ are in bijection with increasing spanning trees $T$ of $G$. Indeed, if $T$ were not increasing, then there is some vertex $i>0$ that is smaller than its parent but larger than all of its descendants. But then $i$ has no outgoing edges in $T$, so there cannot be a nonnegative flow supported on $T$ with net flow 1 at $i$.
		
		Given an increasing spanning tree $T$ of $G$, there is a unique flow $A \in \F_G$ supported on the edges of $T$: we must have that the flow on the edge connecting $i$ to its parent is $\delta_T(i)$. Hence
		\[\operatorname {Ehr}_{q, 1}(\F_G) = \sum_{T \text{ increasing}} \prod_{i=1}^n wt_{q,1}(\delta_T(i)) = \sum_{T \text{ increasing}} \prod_{i=1}^n [\delta_T(i)]_q = I_G(q)\]
		by Proposition~\ref{prop:threshinv}.
	\end{proof}
	
	As a corollary, we can specialize to the case when $G$ is the complete graph $K_{n+1}$. This gives the $t=1$ case of the Haglund-Loehr conjecture (Theorem~\ref{HLhilb}) via Theorem~\ref{thm:1}.

	\begin{corollary}
		We have
		\[\operatorname {Ehr}_{q,1}(\F_{K_{n+1}}) = t_{K_{n+1}}(1,q) = I_{K_{n+1}}(q) = \sum_{T} q^{\inv(T)} = \sum_P q^{\area(P)},\]
		where $T$ ranges over all spanning trees of $K_{n+1}$, and $P$ ranges over all parking functions of length $n$.
	\end{corollary}

	\subsection{General flows}

	We now give a combinatorial formula for $\Ehr_{q,1}(\F_G(\b a))$ for arbitrary $\b a \in \ZZ_{>0}^n$ as a weighted sum over spanning trees over $G$. This formula is analogous to a result by Armstrong et al.\ \cite[Theorem 7.1]{AGHRS} in the case of the complete graph.

	Note that it is straightforward to give a combinatorial formula for $\Ehr_{q,1}(\F_G(\b a))$ as a weighted sum over increasing spanning trees. For a similar result for the complete graph, see Wilson \cite[\S 6]{AW}.
	\begin{proposition} \label{prop:q1}
		Let $G$ be a threshold graph. For any vertex $i>0$, let $\delta_T^{\b a}(i) = \sum_j a_j$, where $j$ ranges over descendants of $i$ (including $i$ itself). Then
		\[\Ehr_{q,1}(\F_G(\b a)) = \sum_{T \text{ increasing}} \prod_{i=1}^n [\delta_T^{\b a}(i)]_q.\]
	\end{proposition}
	\begin{proof}
		As in the proof of Theorem~\ref{thm:t=1}, the only nonzero terms in the sum for $\Ehr_{q,1}(\F_G)$ come from flows supported on increasing spanning trees of $G$. For any such tree, there is a unique flow in $\F_G(\b a)$ supported on it: the flow on the edge connecting $i$ to its parent is $\delta_T^{\b a}(i)$. The result follows easily.
	\end{proof}

	The following theorem converts this formula from a sum over increasing spanning trees of $G$ to a sum over all spanning trees of $G$. For any spanning tree $T$, let $E(T)$ denote the edge set of $T$, $p_T(i)$ denote the parent of vertex $i$, and $\Inv(T)$ denote the set of inversions of $T$.

	\begin{theorem} \label{thm:q1a}
		Let $G$ be a threshold graph and $\b a \in \ZZ_{>0}^n$. 
		\begin{enumerate}[(a)]
			\item Let $\widetilde G$ be the multigraph obtained from $G$ by replacing each edge $(i,j)$ with $a_{\max\{i,j\}}$ parallel edges. If $G$ is connected, then $\Ehr_{q,1}(\F_G(\b a)) = t_{\widetilde G}(1, q)$.
			\item For any spanning tree $T$ of $G$, let
			\[w(T) = \prod_{(i,j) \in E(T)} [a_{\max\{i,j\}}]_q \cdot \prod_{(i,j) \in \Inv(T)} q^{a_{\max\{p_T(i),j\}}}.\]
			Then $\Ehr_{q,1}(\F_G(\b a)) = \sum_T w(T)$, where $T$ ranges over all spanning trees of $G$.
		\end{enumerate}
	\end{theorem} 

	Note that if we set $a_1=a_2=\cdots = a_n=1$, then $w(T) = q^{\inv(T)}$, so we recover Theorem~\ref{thm:t=1}. We will give two proofs of this result. The first is an adaptation of the proof of Theorem~\ref{thm:gessel} above by Gessel in \cite{gessel}. The second uses known properties of the Tutte polynomial.

	\begin{proof}[Proof 1 of Theorem~\ref{thm:q1a}]
		For part (a), let $c_{\widetilde G}(q) = \sum_{H} q^{|E(H)|}$, where $H$ ranges over connected sub-multigraphs of $\widetilde G$. For any such $H$ and any fixed vertex $r$, $H \backslash \{r\}$ decomposes into connected components, yielding an unordered set partition $V_1, \dots, V_k$ of $V \backslash \{r\}$. Let $a(r, V_j)$ denote the total number of edges in $\widetilde G$ from $r$ to a vertex in $V_j$.	Since $H$ must have at least one edge from $r$ to a vertex in $V_j$ for each $j$, and the induced subgraphs $H[V_j]$ are all connected, we have
		\[c_{\widetilde G}(q) = \sum_{V_1, \dots, V_k} \prod_{j=1}^k ((1+q)^{a(r,V_j)}-1) \cdot c_{\widetilde G[V_j]}(q), \tag{$*$}\]
		where the sum ranges over set partitions $V_1, \dots, V_k$ of $V \backslash \{r\}$.

		To prove that $\Ehr_{q,1}(\F_G(\b a)) = (q-1)^{-n}c_{\widetilde G}(q-1) = t_{\widetilde G}(1,q)$, it suffices to show that, for $r=0$, the weighted Ehrhart sum satisfies the appropriate recursion derived from $(*)$, namely
		\[\Ehr_{q,1}(\F_G(\b a)) =  \sum_{V_1, \dots, V_k} \prod_{j=1}^k [a(r,V_j)]_q \cdot \Ehr_{q,1}(\F_{G[V_j]}(\b a[V_j])), \tag{$**$}\]
		where if $V_j = \{i_0, i_1, \dots, i_s\}$ in order, then $\b a[V_j] = (a_{i_1}, \dots, a_{i_s})$. (Note that $G[V_j]$ is still a threshold graph for all $V_j$). If $r = 0$, then $a(0, V_j) = a_{i_0} + \dots + a_{i_s} = \delta_T^{\b a}(i_0)$, where $T$ is any increasing spanning tree of $G$ with a subtree supported on $V_j$. Part (a) now follows easily from Proposition~\ref{prop:q1}. In particular, since $(*)$ is satisfied for all $r$, so must $(**)$ also be satisfied for all $r$.

		We now show part (b) by induction on $n$---in fact, we will show that it holds for any choice of root $r$, not just $r=0$. (Changing the root will usually change the second factor in $w(T)$.)  To see this, let $T$ be a spanning tree of $G$ with subtrees $T_1, \dots, T_k$ on vertex sets $V_1, \dots, V_k$. If $v_j \in V_j$ is a child of the root $r$, then
		\[w(T) = \prod_{j=1}^k w(T_j) [a_{\max\{r,v_j\}}]_q \prod_{\substack{i \in V_j\\i<v_j}} q^{a_{\max\{r, i\}}}.\]
		By induction, for a fixed $v_j$, $\sum_{T_j} w(T_j) = \Ehr_{q,1}(\F_{G[V_j]}(\b a[V_j]))$, which does not depend on $v_j$. Moreover, as $v_j$ ranges over vertices in $V_j$ adjacent to $r$ (noting that any $i<v_j$ is also adjacent to $r$ since $G$ is a threshold graph),
		\[\sum_{v_j} [a_{\max \{r, v_j\}}]_q \prod_{\substack{i \in V_j\\i < v_j}} q^{a_{\max\{r,i\}}} = [\textstyle\sum_{v_j} a_{\max \{r, v_j\}}]_q = [a(r, V_j)]_q.\]
		Hence summing over all spanning trees $T$,
		\begin{align*}
			\sum_T w(T) &= \sum_{V_1, \dots, V_k} \prod_{j=1}^k \sum_{v_j} \sum_{T_j} w(T_j)  [a_{\max\{r,v_j\}}]_q \prod_{\substack{i \in V_j\\i<v_j}} q^{a_{\max\{r, i\}}}\\
			&=\sum_{V_1, \dots, V_k} \prod_{j=1}^k  [a(r, V_j)]_q \Ehr_{q,1}(\F_{G[V_j]}(\b a[V_j]))\\
			&= \Ehr_{q,1}(\F_G(\b a)). \qedhere
		\end{align*}
	\end{proof}

	For the second proof, we recall the following properties of the Tutte polynomial (see \cite[Ch. X]{bollobas}). The first is that the Tutte polynomial satisfies the following \emph{deletion-contraction} recurrence.
	\begin{proposition}
		Let $G$ be a multigraph and $e$ an edge of $G$.
		\begin{enumerate}[(a)]
			\item If $e$ is not a bridge or loop of $G$, then $t_G(x,y) = t_{G-e}(x,y) + t_{G/e}(x,y)$, where $G-e$ and $G/e$ are obtained from $G$ by removing edge $e$ and contracting edge $e$, respectively.
			\item If $e$ is a bridge of $G$, then $t_G(x,y) = x t_{G/e}(x,y)$.
			\item If $e$ is a loop of $G$, then $t_G(x,y) = y t_{G-e}(x,y)$.
		\end{enumerate}
	\end{proposition}

	The second is that the Tutte polynomial can be described in terms of internal and external activity as follows. Fix a total order $\prec$ on the edges of $G$. Given a spanning tree $T$, we call an edge $e \in T$ \emph{internally active} if $e$ is the smallest edge of $G$ joining the two connected components of $T - e$. We call an edge $e \notin T$ \emph{externally active} if $e$ is the smallest edge in the unique cycle of $T \cup \{e\}$. Then the \emph{internal} and \emph{external} activities $ia(T)$ and $ea(T)$ are the total number of internally and externally active edges of $T$, respectively.

	\begin{proposition}
		Let $G$ be a multigraph. Then $t_G(x,y) = \sum_T x^{ia(T)}y^{ea(T)}$, where $T$ ranges over spanning trees of $G$. (This does not depend on the choice of total order $\prec$.)
	\end{proposition}

	We are now ready to give a second proof of Theorem~\ref{thm:q1a}. Although one can use the method of part (b) below to prove part (a) as well via Proposition~\ref{prop:q1}, we present a proof using the deletion-contraction recurrence since a similar recurrence will appear in the proof of Theorem~\ref{thm:q-1}.

	\begin{proof}[Proof 2 of Theorem~\ref{thm:q1a}]
		For (a), let $m$ be the largest neighbor of vertex $n$ in $G$. We will use the deletion-contraction recurrence on each of the $a_n$ edges of $\widetilde G$ from $n$ to $m$. At most one of these edges can be contracted, and all subsequent edges become loops. Let $\widetilde G'$ be the graph obtained by contracting any one of these edges and then removing all loops, and let $\widetilde G''$ be the graph obtained by deleting all of these edges. Then we get
		\[t_{\widetilde G}(1,q) = \begin{cases}
			[a_n]_q\cdot t_{\widetilde G'}(1,q) + t_{\widetilde G''}(1,q) & \text{ if $m>0$,}\\
			[a_n]_q\cdot t_{\widetilde G'}(1,q) & \text{ if $m=0$.}
		\end{cases}\]
		(When $m=0$, the last edge from $n$ to $m$ is a bridge so it cannot be deleted.) Note that for $j<m$, the number of edges in $\widetilde G'$ from $m$ to $j$ is $a_m+a_n$; hence $\widetilde G'$ comes from the threshold graph $G'$ ($G$ with vertex $n$ removed) by multiplying edges according to the flow vector $\b a' = (a_1, \dots, a_{m-1}, a_m+a_n, a_{m+1}, \dots, a_{n-1})$ if $m>0$, and $\b a' = (a_1, \dots, a_{n-1})$ if $m=0$. Likewise $\widetilde G''$ comes from the threshold graph $G''$ ($G$ with edge $(n, m)$ removed) with flow vector $\b a$.

		In fact, $\Ehr_{q,1}(\F_G(\b a))$ satisfies the same recurrence, that is,
		\[\Ehr_{q,1}(\F_G(\b a)) = \begin{cases}
			[a_n]_q \cdot \Ehr_{q,1}(\F_{G'}(\b a')) + \Ehr_{q,1}(\F_{G''}(\b a)) & \text{ if $m>0$,}\\
			[a_n]_q \cdot \Ehr_{q,1}(\F_{G'}(\b a')) & \text{ if $m=0$.}
		\end{cases}\]
		Indeed, as in Proposition~\ref{prop:q1}, we need only consider flows supported on spanning trees $T$ of $G$. If $(n,m) \in T$, then it must support a flow of size $a_n$, which changes the net flow at $m$ on the rest of $T$ from $a_m$ to $a_m+a_n$---this gives the first term in the sum. If $m>0$ and $(n,m) \notin T$, then we get the second term in the sum. It follows that $\Ehr_{q,1}(\F_G(\b a)) = t_{\widetilde G}(1,q)$ by induction on the number of edges of $G$ (the base case with one edge is trivial).

		For (b), we again prove the claim for any root $r$ by induction on $n$. We need to show that
		\[\sum_T wt(T) = \sum_{\widetilde{T}} q^{ea(\widetilde T)},\tag{$\dagger$}\]
		where $T$ and $\widetilde T$ range over spanning trees of $G$ and $\widetilde G$, respectively. (Recall that the right side does not depend on the choice of total order.) Fix a set partition $V_1, \dots, V_k$ of $V \backslash \{r\}$ and vertices $v_j \in V_j$ adjacent to $r$. Then restrict both sides of $(\dagger)$ to trees $T$ and $\widetilde T$ such that the $v_j$ are the children of $r$, and the $V_j$ are the vertex sets supporting the corresponding subtrees $T_j$ and $\widetilde T_j$. The left hand side then becomes, by induction,
		\[\prod_{j=1}^k \sum_{T_j}wt(T_j) \cdot  [a_{\max\{r,v_j\}}]_q \prod_{\substack{i \in V_j\\i<v_j}}q^{a_{\max\{r,i\}}} = \prod_{j=1}^k \sum_{\widetilde T_j}q^{ea(\widetilde T_j)} \cdot  [a_{\max\{r,v_j\}}]_q \prod_{\substack{i \in V_j\\i<v_j}}q^{a_{\max\{r,i\}}}.\]
		We claim this is also what the right hand side of $(\dagger)$ becomes.

		Choose any total order on the edges of $\widetilde G$ that starts with all edges between $r$ and $0$, then all edges between $r$ and $1$, and so forth. (The edges not containing $r$ can be in any order after that.) No edges between distinct $V_i$ and $V_j$ are externally active, so the external activity of $\widetilde T$ is the sum of the external activities of its subtrees $\widetilde T_j$ plus the number of externally active edges containing $r$. Of the $a_{\max\{r,v_j\}}$ edges from $r$ to $v_j$, any number from $0$ to $a_{\max\{r, v_j\}}-1$ are externally active depending on which parallel edge lies in $\widetilde T$. For $i \in V_j \backslash \{v_j\}$, all edges from $r$ to $i$ are externally active if $i<v_j$, otherwise none are. The result follows.
	\end{proof}

	\begin{remark}
		One special case worth noting is when $q=t=1$. In this case, Theorem~\ref{thm:q1a} implies that $\Ehr_{1,1}(\F_G(\b a))$ is the total weight of all spanning trees $T$ of $G$, where the weight of any edge $(i,j)$ is $a_{\max\{i,j\}}$. Thus $\Ehr_{1,1}(\F_G(\b a))$ can be expressed as a determinant using the Matrix-Tree Theorem. In fact, one can show that this determinant factors into linear factors. For instance, when $G = K_{n+1}$, $\Ehr_{1,1}(\F_{K_{n+1}}(\b a)) = \det M$, where
		\[M = \begin{bmatrix}
			a_1+a_2+\cdots + a_n & -a_2 & -a_3 & \cdots & -a_n\\
			-a_2 & 2a_2 + a_3 + \cdots + a_n &-a_3 & \cdots & -a_n\\
			-a_3& -a_3 & 3a_3 + a_4 + \cdots + a_n& \cdots & -a_n\\
			\vdots & \vdots & \vdots & \ddots & \vdots\\
			-a_n & -a_n & -a_n & \cdots & na_n
		\end{bmatrix}.\]
		Multiplying the $i$th row by $i+1$ and adding all the lower rows to it for $i=1, \dots, n-1$ yields a lower triangular matrix, so one can easily recover the result of Armstrong et al.\ \cite{AGHRS} that
		\[\Ehr_{1,1}(\F_{K_{n+1}}(\b a)) = a_n \cdot \prod_{i=1}^{n-1} ((i+1)a_i + \textstyle\sum_{j=i+1}^n a_j).\]
		We will see a generalization of this product formula for general threshold graphs $G$ later in Section~\ref{sec:qq^{-1}} when we compute $\Ehr_{q,q^{-1}}(\F_G(\b a))$.
	\end{remark}

\section{Calculating the $(q,0)$-Ehrhart function}
\label{sec:q0}

	In this section, we give a product formula for the weighted Ehrhart function of the flow polytope $\F_G(\b a)$ when $G$ is a threshold graph and $t=0$. In particular, when $q=1$ and $t=0$, the weighted Ehrhart function evaluates to the number of increasing spanning trees of $G$, or equivalently the number of maximal $G$-parking functions.
	
	For a threshold graph $G$, let $\od_i = \min\{d_i, i\}$ be the outdegree of vertex $i$, that is, the number of vertices $j$ adjacent to $i$ with $j<i$. It should be noted that these outdegrees are closely related to the number of increasing spanning trees of $G$.
	
	\begin{proposition}
		Let $G$ be a threshold graph. The number of increasing spanning trees of $G$ is $\prod_{i=1}^n \od_i$.
	\end{proposition}
	\begin{proof}
		Each vertex $i>0$ has a choice of $\od_i$ vertices to be its parent.
	\end{proof}
	
	We now state the main result of this section. Observe that when $t=0$, the weights specialize to
	\[wt_{q,0}(A) = (q-1)^{\#\{a_{ij}>0\}-n}\cdot \prod_{n \geq i > j \geq 0} wt_{q,0}(a_{ij})\text{, \quad where}\quad
	wt_{q,0}(b) = \begin{cases}q^{b-1}&\text{if $b>0$,}\\1&\text{if $b=0$.}\end{cases}\]
	
	\begin{theorem}\label{thm:t=0}
		Let $G$ be a threshold graph and $\b a \in \ZZ_{>0}^n$. Then
		\[\operatorname {Ehr}_{q,0}(\F_G(\b a)) = \prod_{i=1}^n q^{\od_i(a_i-1)} [\od_i]_q.\]
	\end{theorem}
	Note that when $a_1=a_2=\dots = a_n=1$, this formula gives a $q$-analogue for the number of increasing spanning trees on $G$.

	We will first need the following lemma.
	\begin{lemma} \label{lem:t=0}
		For integers $c\geq 1$ and $k \geq 1$, let $\Delta = \Delta(k,c) = \{(b_0, \dots, b_{k-1}) \mid \sum b_i = c\}$. For any $B \in \Delta \cap \ZZ^{k}$, define
		\[wt_{q,0}(B) = (q-1)^{\#\{b_i > 0\}-1} \prod_i wt_{q,0}(b_i).\]
		Then
		\[\sum_{B \in \Delta \cap \ZZ^{k}} q^{b_1+2b_2+\cdots + (k-1)b_{k-1}}wt_{q,0}(B) = q^{k(c-1)}[k]_q.\]
	\end{lemma}

	\begin{proof}	
		We induct on $k$. When $k=1$, $\Delta$ has a single point $c$, and both sides equal $q^{c-1}$. We therefore assume $k>1$.

		For $B = (b_0, \dots, b_{k-1})$, write $B' = (b_0, \dots, b_{k-2})$. Letting $b=b_{k-1}$, we have the decomposition
		\[\Delta \cap \ZZ^k = \bigcup_{b=0}^{c} (\Delta(k-1,c-b) \times \{b\}) \cap \ZZ^k\]
		(where $\Delta(k,0)$ is the set containing the single point $0 \in \ZZ^k$). Since
		\[wt_{q,0}(B) = \begin{cases}
		wt_{q,0}(B')&\text{if $b=0$,}\\
		wt_{q,0}(B')(q^b-q^{b-1})&\text{if $0<b<c$,}\\
		q^{c-1}&\text{if $b=c$,}
		\end{cases}\]
		we have by the inductive hypothesis that, for fixed $b$,
		\begin{align*}
		\sum_{\substack{B \in \Delta\cap \ZZ^k\\b_{k-1}=b} } q^{b_1+2b_2+\dots + (k-1)b_{k-1}} wt_{q,0}(B) =
		\begin{cases}
			q^{(k-1)(c-1)}[k-1]_q&\text{if $b=0$,}\\
			q^{(k-1)(c-b-1)}[k-1]_q \cdot q^{(k-1)b}(q^b-q^{b-1})&\text{if $0<b<c$,}\\
			q^{(k-1)c}\cdot q^{c-1}&\text{if $b=c$.}
		\end{cases}
		\end{align*}
		Summing over all $b$ gives the telescoping sum
		\begin{align*}
			q^{(k-1)(c-1)}[k-1]_q + \sum_{b=1}^{c-1} (q^{(k-1)(c-1)}[k-1]_q)(q^b-q^{b-1}) &+ q^{ck-1}\\
			&= q^{(k-1)(c-1)}[k-1]_q \cdot q^{c-1} + q^{ck-1}\\
			&= q^{k(c-1)}([k-1]_q+q^{k-1})\\
			&= q^{k(c-1)}[k]_q. \qedhere
		\end{align*}
	\end{proof}

	We now proceed with the proof of the theorem.

	\begin{proof}[Proof of Theorem~\ref{thm:t=0}]
		We may assume that $G$ is connected for both sides to be nonzero. We induct on $n$. When $n=1$, we must have $\od_1=1$, so both sides equal $wt_{q,0}(a_1) = q^{a_1-1}$.
		
		Now assume $n > 1$.	The vertex $n$ is adjacent to vertices $0, 1, \dots, \od_n-1$. For any lattice point $A \in \F_G(\b a)$, write 
		\[B = (a_{n0}, a_{n1}, \dots, a_{n,\od_n-1}) = (b_0, b_1, \dots, b_{\od_n-1})\]
		so that $B$ ranges over all lattice points in $\Delta(\od_n,a_n)$. For fixed $B$, the remaining flow $A'$ on the graph $G'$ obtained from $G$ by removing vertex $n$ lies in $\F_{G'}(\b a')$, where
		\[\b a' = (a_1, a_2, \dots, a_{n-1}) + (b_1, \dots, b_{\od_n-1}, 0, \dots, 0) = (a_1', \cdots, a_{n-1}').\]
		Note that the outdegree of a vertex $i<n$ in $G'$ is still $\od_i$. Since $wt_{q,0}(A) = wt_{q,0}(B) \cdot wt_{q,0}(A')$, we have by induction and Lemma~\ref{lem:t=0} that
		\begin{align*}
			\sum_{A \in \F_G(\b a)\cap \ZZ^E}wt_{q,0}(A) &= \sum_{B \in \Delta(\od_n,a_n) \cap \ZZ^{\od_n}} \left(wt_{q,0}(B) \cdot \sum_{A' \in \F_{G'}(\b a')} wt_{q,0}(A')\right)\\
			&= \sum_{B \in \Delta(\od_n,a_n) \cap \ZZ^{\od_n}} \left(wt_{q,0}(B) \cdot \prod_{i=1}^{n-1} q^{\od_i(a_i'-1)}[\od_i]_q\right)\\ 
			&= \sum_{B \in \Delta(\od_n,a_n) \cap \ZZ^{\od_n}}  q^{b_1+2b_2+\cdots + (\od_n-1)b_{\od_n-1}}wt_{q,0}(B)\cdot \prod_{i=1}^{n-1} q^{\od_i(a_i-1)}[\od_i]_q\\ 			
			&= q^{\od_n(a_n-1)}[\od_n]_q \cdot \prod_{i=1}^{n-1} q^{\od_i(a_i-1)}[\od_i]_q\\
			&= \prod_{i=1}^{n} q^{\od_i(a_i-1)}[\od_i]_q.\qedhere
		\end{align*}
	\end{proof}
	
	Specializing to the case when $G = K_{n+1}$ gives the following corollary.
	
	\begin{corollary}
		Let $\b a \in \ZZ^n_{>0}$. Then
		\[\operatorname {Ehr}_{q,0}(\F_{K_{n+1}}(\b a)) = q^{a_1+2a_2+\cdots+na_n - \binom{n+1}{2}} [n]_q!,\]
		where $[n]_q! = [n]_q[n-1]_q \cdots [1]_q$. In particular, $\operatorname {Ehr}_{q,0}(\F_{K_{n+1}}) = [n]_q!$.
	\end{corollary}
	
	\begin{remark}
		The case $\Ehr_{q,0}(\mathcal{F}_{K_{n+1}}(-n,1,\ldots,1)) = [n]_q!$ was known by combining Theorem~\ref{thm:1} with \eqref{case:q0}. There is an elegant proof of this result by Levande \cite{PL} who defined a function $\varphi$ from integer flows on $K_{n+1}$ with netflow $(-n,1,\ldots,1)$ to permutations in $\Sn$ and used a sign-reversing involution to show that $\sum_{A \in \varphi^{-1}(w)} wt_{q,t}(A) = q^{\inv(w)}$. Wilson \cite[\S 5]{AW} extended this involution to the case $\Ehr_{q,0}(\mathcal{F}_{K_{n+1}}({\bf a}))$ where $a_i\in \{0,1\}$. In contrast with these proofs, our proof is inductive and does not use involutions. 
	\end{remark}

\section{Calculating the $(q, q^{-1})$-Ehrhart function} 
\label{sec:qq^{-1}}

	In this section, we give a product formula for the weighted Ehrhart function of the flow polytope $\F_G(\b a)$ when $G$ is a threshold graph and $t=q^{-1}$. When specialized to the case $G = K_{n+1}$, this proves a conjecture of  Armstrong et al. \cite{AGHRS}.

	From Theorem~\ref{thm:t=1}, we know that $\operatorname {Ehr}_{q,q^{-1}}(\F_G(\b a))$ should specialize to the number of spanning trees of $G$ when $q=1$ and $\b a = 1$. In fact, for threshold graphs $G$, there is a simple product formula for the number of spanning trees. Let $c_i = \#\{j \mid d_j \geq i\}$. In other words, $(c_1, c_2, \dots, c_{n})$ is the conjugate partition to $d(G)$.
	
	\begin{proposition} \label{prop:spt-prod}
		Let $G$ be a threshold graph on $0, 1, \dots, n$. The number of spanning trees of $G$ is
		\[c_2c_3 \cdots c_n = \prod_{i \colon 0<i<d_i} (d_i+1) \cdot \prod_{i \colon d_i<i}d_i.\]
	\end{proposition}
	
	This is a direct application of the Matrix-Tree Theorem; see also \cite{chestnut} for a combinatorial proof. Note that when $G$ is the complete graph, we recover Cayley's formula $(n+1)^{n-1}$ for the number of spanning trees of $K_{n+1}$.
	
	We now state the main result of this section. Observe that when $t=q^{-1}$, the weights specialize to
	\[wt_{q, q^{-1}}(A) = (-(1-q)(1-q^{-1}))^{\#\{a_{ij}>0\}-n}\prod_{n\geq i>j \geq 0}wt_{q,q^{-1}}(a_{ij}),\]where
	\[wt_{q,q^{-1}}(b) = \begin{cases}\frac{q^b-q^{-b}}{q-q^{-1}}&\text{if $b>0$,}\\1&\text{if $b=0$.}\end{cases}\]
	Also recall that $\od_i = \min\{d_i,i\}$ is the outdegree of vertex $i$.
	\begin{theorem} \label{thm:q-1}
		Let $G$ be a threshold graph and $\b a \in \ZZ^n_{>0}$. Then
		\[\operatorname {Ehr}_{q,q^{-1}}(\F_G(\b a)) = q^{-F}\prod_{i=1}^n b_i(q),\]
		where $F = \sum_{i=1}^n \od_ia_i -n$ and 
		\[b_i(q) = \begin{cases}
			[(i+1)a_i+\sum_{j=i+1}^{d_i}a_j]_q&\text{if }d_i>i,\\
			[a_i]_{q^{i+1}}&\text{if }d_i = i,\\
			[a_i]_{q^{d_i+1}}[d_i]_q&\text{if }d_i  < i.
		\end{cases}\]
	\end{theorem}
	
	Before we get to the proof, note what happens when we specialize $a_1=a_2=\cdots = a_n=1$. In this case,
	\[b_i(q) = \begin{cases}[d_i+1]_q & \text{if $d_i>i$,}\\1 & \text{if $d_i=i$,}\\ [d_i]_{q}&\text{if $d_i<i$,}\end{cases}\]
	so Theorem~\ref{thm:q-1} gives a $q$-analogue of Proposition~\ref{prop:spt-prod} in this case.
	
	\begin{proof}[Proof of Theorem~\ref{thm:q-1}]
		We may assume $G$ is connected and induct on $n$ and $d_n$. When $n=1$, we have $d_1=1$, and $wt_{q,q^{-1}}(a_1) = q^{1-a_1}[a_1]_{q^2}$, so assume $n > 1$.
		
		If $d_n = 1$, let $G'$ be the threshold graph obtained by removing vertex $n$. Then any flow $A \in \F_G(\b a)$ can be obtained from a flow in $\F_{G'}(a_1, \dots, a_{n-1})$ by adding the vertex $n$ and a single edge with flow $a_n$ from $n$ to $0$. Hence
		\[\Ehr_{q,q^{-1}}(\F_G(\b a)) = \Ehr_{q,q^{-1}}(\F_{G'}(a_1, \dots, a_{n-1})) \cdot wt_{q,q^{-1}}(a_n) = q^{-F'} q^{1-a_n}[a_n]_{q^2} \prod_{i=1}^{n-1} b_i'(q),\]
		where $b_i'$ and $F'$ are the corresponding values of $b_i$ and $F$ for $G'$. But $b_i'(q) = b_i(q)$ for $i<n$, $b_n(q) = [a_n]_{q^2}[1]_q = [a_n]_{q^2}$, and $F = F' + a_n-1$, so the right side is $q^{-F}\prod_{i=1}^n b_i(q)$, as desired.
		
		Now suppose $d_n-1=m>0$. Then vertex $n$ is adjacent to $0, 1, \dots, m$. Let $G'$ be the threshold graph obtained by removing vertex $n$, and let $G''$ be the threshold graph obtained from $G$ by removing only the edge from $n$ to $m$. Choose any $A \in \F_G(\b a)$, and let $k = a_{n, m}$.
		\begin{itemize}
			\item If $k=0$, then $A \in \F_{G''}(\b a)$.
			\item If $k = a_n$, then $A$ can be obtained from a flow in $\F_{G'}(a_1, \dots, a_m+a_n, \dots, a_{n-1})$ by adding vertex $n$ and flow $a_n$ from $n$ to $m$.
			\item If $0<k<a_n$, then $A$ can be obtained from a flow in $\F_{G''}(a_1, \dots, a_m+k, \dots, a_{n-1}, a_n-k)$ by adding flow $k$ from $n$ to $m$.
		\end{itemize}
		It follows that
		\begin{align*}
			\Ehr_{q,q^{-1}}&(\F_G(\b a))= \Ehr_{q,q^{-1}}(\F_{G''}(\b a)) + \Ehr_{q,q^{-1}}(\F_{G'}(a_1, \dots, a_{m}+a_n, \dots, a_{n-1})) \cdot wt_{q,q^{-1}}(a_n) \\
			&{}- \sum_{k=1}^{a_n-1} (1-q)(1-q^{-1}) \Ehr_{q,q^{-1}}(\F_{G''}(a_1, \dots, a_{m} + k, \dots, a_{n-1}, a_n-k))\cdot wt_{q,q^{-1}}(k).
		\end{align*}
		By induction, we may expand each of the terms on the right side. First note that for any of the terms involving $G''$, the corresponding value of $F$ is, for any $k=0, \dots, a_n-1$,
		\[F-a_nd_n + (a_n-k)(d_n-1) + km = F-a_n,\]
		while for the $G'$ term, the corresponding value of $F$ is $F-d_na_n + a_nm +1 = F-a_n+1$.
		
		Next observe that for $i \neq n, m$, the value of $b_i(q)$ is the same in all terms. Indeed, this is clear by the definition of $b_i(q)$ if $d_i \leq i$, so assume $d_i>i$. Then if $i<m$, vertex $i$ is adjacent to $n$, so $b_i(q) = [(i+1)a_i + \sum_{j>i} a_j]_q$, which is the same in all terms. If instead $m<i<n$, then $i$ is not adjacent to $n$, so $b_i(q) = [(i+1)a_i + \sum_{j=i+1}^{d_i}a_j]_q$ does not involve either $a_n$ or $a_{m}$, so it is also unchanged. It follows that we need only compare $b_n(q)$ and $b_m(q)$ for each of the terms.
		
		Therefore to prove the theorem, it suffices to show that, if $d_n<n$ (so $d_m = n > m+1$),
		\begin{align*}
			[a_n&]_{q^{d_n+1}}[d_n]_q [d_na_m + \textstyle\sum_{j=m+1}^n a_j]_q = \\
			&q^{a_n} [a_n]_{q^{d_n}}[d_n-1]_q[d_na_m + \textstyle\sum_{j=m+1}^{n-1} a_j]_q+q^{a_n-1}wt_{q,q^{-1}}(a_n)[d_n(a_m+a_n) + \textstyle\sum_{j=m+1}^{n-1} a_j]_q\\
			&{}-q^{a_n}(1-q)(1-q^{-1})\sum_{k=1}^{a_n-1}wt_{q,q^{-1}}(k)[a_n-k]_{q^{d_n}}[d_n-1]_q[d_n(a_m+k)+\textstyle\sum_{j=m+1}^{n-1}a_j]_q,
		\end{align*}
		while if $d_n=n$,
		\begin{align*}
			[a_n]_{q^{n+1}}&[na_{n-1}+ a_n]_q = \\
			&q^{a_n} [a_n]_{q^{n}}[n-1]_q[a_{n-1}]_{q^n}+q^{a_n-1}wt_{q,q^{-1}}(a_n)[a_{n-1}+a_n]_{q^n}\\
			&{}-q^{a_n}(1-q)(1-q^{-1})\sum_{k=1}^{a_n-1}wt_{q,q^{-1}}(k)[a_n-k]_{q^{n}}[n-1]_q[a_{n-1}+k]_{q^n}.
		\end{align*}		
		Both of these follow from Lemma~\ref{lemma q} below: the first follows by letting $a=a_n$, $d=d_n$, and $z = d_na_m + \sum_{j=m+1}^{n-1}a_j$, while the second follows by letting $a=a_n$, $d=n$, and $z=na_{n-1}$ and dividing both sides by $[n]_q$ (using the fact that $[nx]_q = [x]_{q^n}[n]_q$).
	\end{proof}
	
	\begin{lemma} \label{lemma q}
		Let $a$, $d$, and $z$ be positive integers. Then
		\begin{align*}
			[a]_{q^{d+1}}[d]_q[z+a]_q = {}&q^a[a]_{q^d}[d-1]_q[z]_q + q^{a-1}wt_{q,q^{-1}}(a)[z+da]_q \\
			&{}- q^a(1-q)(1-q^{-1})\sum_{k=1}^{a-1}wt_{q,q^{-1}}(k)[a-k]_{q^d}[d-1]_q[z+dk]_{q}.
		\end{align*}
	\end{lemma}
	\begin{proof}
		We compute
		\begin{align*}
			f(x) &= 1-q(1-q)(1-q^{-1})\sum_{k \geq 1}q^{k}[k]_{q^d}[d-1]_qx^k \\
			&= 1+(1-q)(1-q^{d-1})\sum_{k \geq 1}[k]_{q^d}(qx)^k\\
			&=1+\frac{(1-q)(1-q^{d-1})qx}{(1-qx)(1-q^{d+1}x)}\\
			&=\frac{(1-q^2x)(1-q^dx)}{(1-qx)(1-q^{d+1}x)}
		\end{align*}
		and
		\begin{align*}
			g(x) &= \frac{-[z]_q}{q(1-q)(1-q^{-1})} + \sum_{k \geq 1}q^{k-1}wt_{q,q^{-1}}(k)[z+dk]_qx^k\\
			&= \frac{1-q^z}{(1-q)^3} + \sum_{k \geq 1} [k]_{q^2}[z+dk]_qx^k\\
			&= \bar g(x) -q^z \bar g(q^dx),
		\end{align*}
		where
		\begin{align*}
			\bar g(x) &= \frac{1}{(1-q)^3} + \frac{1}{1-q} \cdot \sum_{k \geq 1} [k]_{q^2}x^k\\
			&= \frac{1}{(1-q)^3} + \frac{x}{(1-q)(1-x)(1-q^2x)}\\
			&= \frac{(1-qx)^2}{(1-q)^3(1-x)(1-q^2x)}.
		\end{align*}

		For $a \geq 1$, the desired right hand side is the coefficient of $x^a$ in $f(x)g(x)=f(x)\bar g(x) - q^z f(x)\bar g(q^dx)$. Since
		\[\sum_{a \geq 1}[a]_{q^{d+1}}[d]_q[z+a]_qx^a = h(x) - q^z h(qx),\]
		where 
		\[
		h(x) = \sum_{a \geq 1}\frac{[a]_{q^{d+1}}[d]_qx^a}{1-q}\\
		= \frac{(1-q^d)x}{(1-q)^2(1-x)(1-q^{d+1}x)},
		\]
		it suffices to check that the difference between the two sides,
		\[(f(x)\bar g(x) - h(x))-q^z(f(x) \bar g(q^dx) - h(qx)),\]
		is independent of $x$. Indeed, we will show that
		\[f(x)\bar g(x) - h(x) = f(x)\bar g(q^dx)-h(qx) = \frac{1}{(1-q)^3}.\]
		This is straightforward:
		\begin{align*}
			f(x)\bar g(x)-h(x) &= \frac{(1-qx)(1-q^dx)}{(1-q)^3(1-x)(1-q^{d+1}x)} -\frac{(1-q^d)x}{(1-q)^2(1-x)(1-q^{d+1}x)}\\
			&= \frac{1}{(1-q)^3},
		\end{align*}
		and 
		\[f(x)\bar g(q^dx) =\frac{(1-q^2x)(1-q^{d+1}x)}{(1-q)^3(1-qx)(1-q^{d+2}x)} = f(qx)\bar g(qx),\]
		so $f(x)\bar g(q^dx)-h(qx) = f(qx)\bar g(qx)-h(qx) = \frac{1}{(1-q)^3}$ as well.	
	\end{proof}

	If we specialize to the case $G=K_{n+1}$, we arrive at the following corollary, conjectured by Armstrong et al. in \cite[Conjecture 7.1]{AGHRS}.
	\begin{corollary} \label{cor:q-1}
		Let $\b a \in \ZZ^n_{>0}$. Then
		\[\operatorname {Ehr}_{q,q^{-1}}(\F_{K_{n+1}}(\b a)) = q^{-F}[a_n]_{q^{n+1}}\prod_{i=1}^{n-1} [(i+1)a_i+\textstyle\sum_{j=i+1}^{n}a_j]_q,\]
		where $F = \sum_{i=1}^n ia_i -n$.
	\end{corollary}

\section{About the $(q, t)$-Ehrhart function} \label{sec:conj}

	In this section we look at the weighted Ehrhart series of the flow polytope $\F_G(-n,1,\ldots,1)$ when $G$ is a threshold graph with $n+1$ vertices. 

	\subsection{Conjectured $q,t$-positivity}

	By Haglund's result \cite{Hag}, the weighted Ehrhart series $\Ehr_{q,t}(\F_{K_{n+1}}(-n,1,\ldots,1))$ is the bigraded Hilbert series of the space of diagonal harmonics, so it must lie in $\mathbf{N}[q,t]$. By Example~\ref{ex:negqtGraph}, the polynomial $\Ehr_{q,t}(\F_G)$ for other graphs $G$ sometimes has negative coefficients. However, experimentation suggests some positivity properties of the polynomials $\Ehr_{q,t}(\F_G)$ for threshold graphs $G$ and netflow $(-n,1,\ldots,1)$. 

	\begin{conjecture} \label{conj1}
	Let $G$ be a threshold graph with $n+1$ vertices. Then
	\[
		\Ehr_{q,t}\bigl( \F_G(-n,1,\ldots,1)\bigr) \in \mathbf{N}[q,t].
	\]
	\end{conjecture}
	This conjecture has been verified up to $n=9$. 

	In \cite{GHX}, it was conjectured that $\Ehr_{q,t}(\F_{K_{n+1}}(\a)) \in \mathbf{N}[q,t]$ for integral netflows $\a$ satisfying $a_1\geq a_2 \geq \cdots \geq a_n \geq 0$. The analogous $q,t$-positivity statement for threshold graphs does not hold even though Theorem~\ref{thm:q-1} gives product formulas when $t=q^{-1}$.

	\begin{example}
		For the threshold graph $G$ with degree sequence $(3,3,2,2)$ and $\a = (-9,3,3,3)$, there are $16$ integral flows, and we have that
		\[
			\Ehr_{q,t}(\F_{G}(\a)) = q^{12} + q^{11}t + q^{10}t^2 + \dots + 2q^4t^3 + 2q^3t^4 -
			q^3t^3 \not\in \mathbb{N}[q,t].
		\]
	\end{example}

	Along with Theorem~\ref{thm:t=1}, Conjecture~\ref{conj1} suggests that there may be some statistic $\stat(\cdot)$ on spanning trees $T$ of $G$ or on $G$-parking functions such that $\sum_{T} q^{\inv(T)}t^{\stat(T)}$ equals $\Ehr_{q,t}(\F_G)$. We have so far been unable to find such a statistic (see Section~\ref{subsec:pmaj}).

	A spanning tree $T$ of a connected threshold graph $G$ with $n+1$ vertices is also a spanning tree of the complete graph. A stronger positivity result would be that each monomial $q^{\inv(T)}t^{\stat(T)}$ in $\Ehr_{q,t}(\F_G)$ appeared also in $\Ehr_{q,t}(\F_{K_{n+1}})$. Calculations up to $n=9$ suggest that this is also the case.

	\begin{conjecture} \label{conj:knminusg}
		Let $G$ be a threshold graph with $n+1$ vertices. Then
		\[
			\Ehr_{q,t}(\F_{K_{n+1}})- \Ehr_{q,t}(\F_G) \in \mathbf{N}[q,t].
		\]
	\end{conjecture}

	The above computations suggest to check positivity of differences of $(q,t)$-Ehrhart functions between a threshold graph and a subgraph that is also a threshold graph. Let $\mathcal{P}_n$ be the poset of connected threshold graphs with vertices $0,1,\ldots,n$, where $H \preceq G$ if $H$ is a subgraph of $G$. This poset is isomorphic to the poset of shifted Young diagrams (or partitions with distinct parts) contained in $(n-1, n-2, \dots, 0)$, ordered by inclusion. For example the Hasse diagrams of the posets $\mathcal{P}_3$ and $\mathcal{P}_4$ are the following:

	\begin{center}
		\includegraphics[scale=0.6]{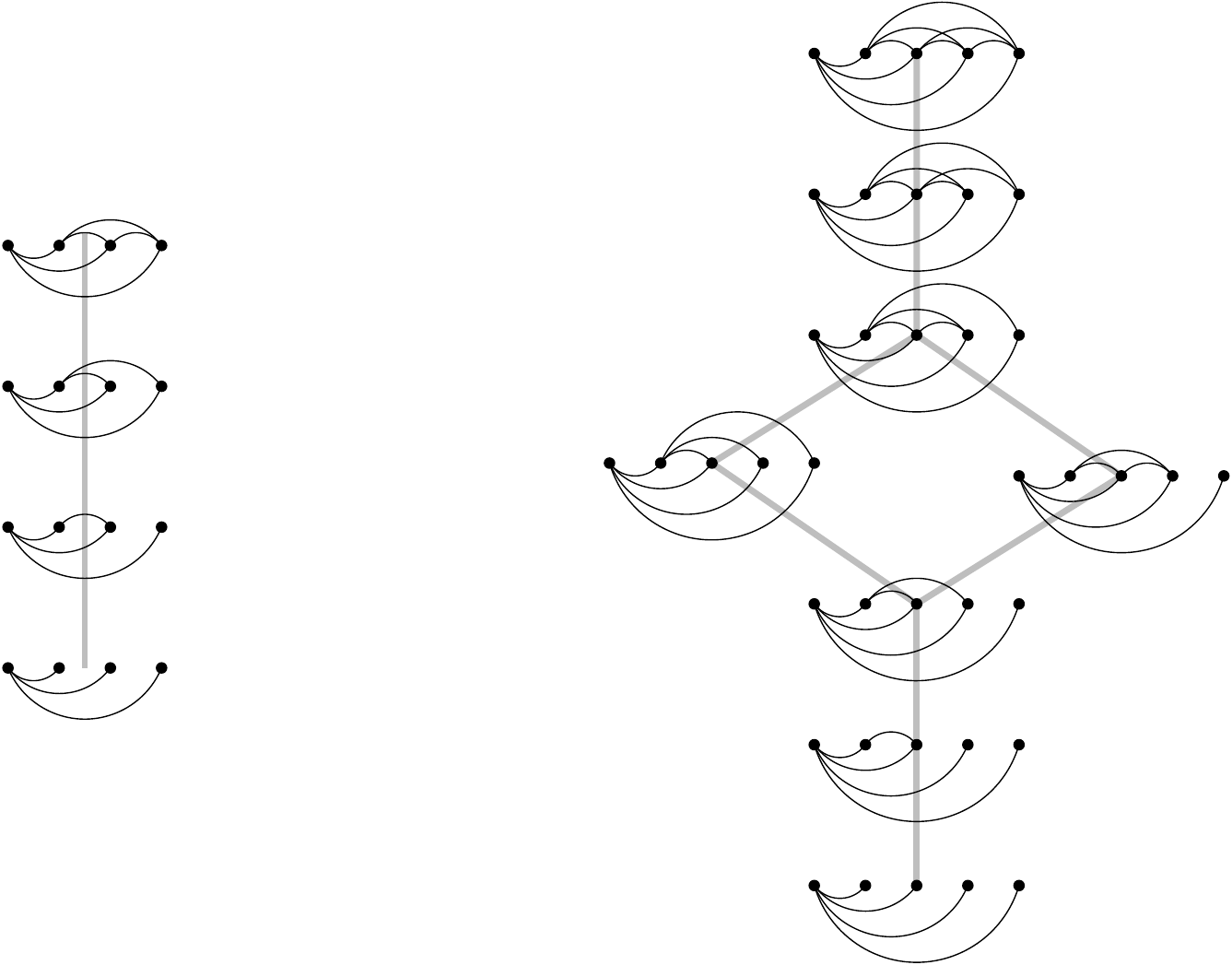}
	\end{center}

	Calculations up to $n=9$ suggest positivity of differences of $(q,t)$-Ehrhart functions along the cover relations of this poset. 

	\begin{conjecture} \label{conj:poset}
		For threshold graphs $G\succeq H$ in $\mathcal{P}_n$,
		\[
			\Ehr_{q,t}(\F_{G})- \Ehr_{q,t}(\F_H) \in \mathbf{N}[q,t].
		\]
	\end{conjecture} 

	Note that Conjecture~\ref{conj:poset} implies Conjecture~\ref{conj:knminusg}. 

	Lastly, one might try to use the poset structure of $\mathcal{P}_n$ to refine further each $(q,t)$-Ehrhart series in the following way. For a threshold graph $G$ in $\mathcal{P}_n$, let 
	\[
		S_{q,t}(G)=\sum_{H\preceq G} \mu(H,G) \Ehr_{q,t}(\F_H),
	\]
	where $\mu(\cdot,\cdot)$ is the M\"obius function of $\mathcal{P}_n$. By M\"obius inversion we then have that 
	\[
		\Ehr_{q,t}(\F_G) = \sum_{H\preceq G} S_{q,t}(H).
	\]
	One might hope that $S_{q,t}(G)$ is $q,t$-positive in general, but this is not the case.
	\begin{example}
		Let $G$ be the threshold graph $G$ with degree sequence $(6,6,6,6,5,5,4)$. Then
		\[S_{q,t}(G) = q^{13} + q^{12}t + q^{11}t^2 + \dots + q^3t^2 + q^2t^3 - q^2t^2 \notin \mathbf{N}[q,t].\]
	\end{example}
	This shows that if the statistic $\stat(T)$ exists such that $\Ehr_{q,t}(\F_G) = \sum_T q^{\inv(T)}t^{\stat(T)}$, then it must depend on the underlying threshold graph $G$.

	\subsection{Positivity with Gorsky--Negut weight}

	One could explore generalizations to other positivity results for Tesler matrices. The \emph{alternant} $DH_n^{\varepsilon}$ is a certain subspace of $DH_n$ of dimension $\frac{1}{n+1}\binom{2n}{n}$ \cite{gh3}, the $n$th Catalan number. The bigraded Hilbert series of this subspace is the $q,t$-Catalan number $C_n(q,t)$. Gorsky and Negut \cite{GN} expressed $C_n(q,t)$ as a different weighted sum over the integral flows of $\F_{K_{n+1}}(-n,1,\ldots,1)$ (see Remark~\ref{rem:flow2Tesler} for translating from integral flows to Tesler matrices).

	\begin{theorem}[Gorsky--Negut \cite{GN}]
		\[
			C_n(q,t) = \sum_{A \in \F_{K_{n+1}} \cap \ZZ^E} wt'_{q,t}(A),
		\]
		where 
		\begin{equation} \label{eq:GNweight}
			wt'_{q,t}(A) = \prod_{\substack{i>1\\a_{i,i-1}>0}} ( wt_{q,t}(a_{i,i-1}+1) -
			wt_{q,t}(a_{i,i-1}) ) \prod_{\substack{i-1>j>0\\ a_{i,j}>0}} \left(-(1-t)(1-q) wt_{q,t}(a_{i,j})\right),
		\end{equation}
		for $wt_{q,t}(b)$ as defined in \eqref{eq:wt}.
	\end{theorem}

	In contrast with the evidence for Conjecture~\ref{conj1}, this weighted sum does not necessarily give a polynomial in $\mathbf{N}[q,t]$ for threshold graphs.

	\begin{example}
		For the threshold graph $G$ with degree sequence $(3,3,2,2)$, there are four integral flows in $\F_G$ with their respective Tesler matrices (see Remark~\ref{rem:flow2Tesler}): 
		\begin{center}
			\includegraphics{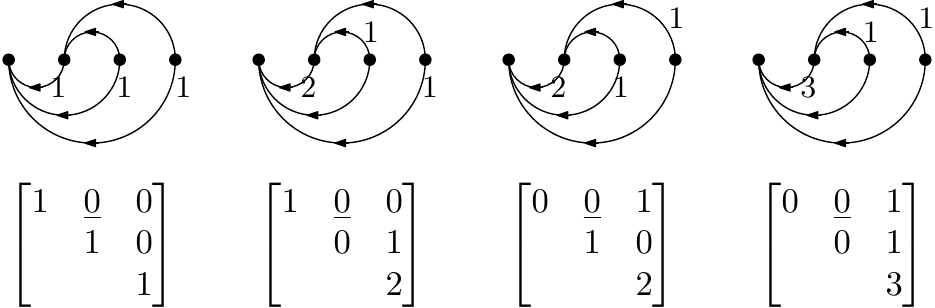}
		\end{center}
		The weighted sum $wt'_{q,t}(A)$ of these flows gives
		\[
			1 + (q+t-1) - (1-q)(1-t) - (q+t-1)(1-q)(1-t) = q^2+2qt+t^2-q^2t-qt^2 \not\in \mathbb{N}[q,t].
		\]
	\end{example}

	The Gorsky--Negut weight \eqref{eq:GNweight} is not as natural for threshold graphs $G$ as it is for the complete graph since $G$ might not contain the edge $(i, i-1)$. Instead, one could consider a weight  
	\[
		wt''_{q,t}(A) = \prod_{\substack{i>1\\a_{i,\bar d_i-1}>0}} ( wt_{q,t}(a_{i,\bar d_i-1}+1) -
		wt_{q,t}(a_{i,\bar d_i-1}) ) \prod_{\substack{\bar d_i-1>j>0\\a_{i,j}>0}} \left(-(1-t)(1-q) wt_{q,t}(a_{i,j})\right),
	\]
	where $\bar d_i - 1$ is the largest neighbor of $i$ less than $i$. Still, summing over this weight does not necessarily yield a polynomial in $\mathbb{N}[q,t]$.

	\begin{example}
		For the threshold graph $G$ with degree sequence $(5,5,5,3,3,3)$, we have $\bar d_2= 2$, and $\bar d_3 = \bar d_4 = \bar d_5 = 3$. There are $81$ integral flows in $\F_G$. The weighted sum of these flows gives
		\[
			\sum_{A \in \F_{G} \cap \ZZ^E} wt''_{q,t}(A) = q^7 + q^6t + q^5t^2 + \cdots + 3q^3t^2 + 3q^2t^3 - q^2t^2 \not\in \mathbb{N}[q,t].
		\]
	\end{example}

	Garsia and Haglund \cite{GHag} gave a weight over certain integral flows in $\F_{K_{n+1}}$ that yields the \emph{Frobenius series} of the space $DH_n$, a certain symmetric function that in the Schur basis has coefficients in $\mathbf{N}[q,t]$ (for more details see \cite[Ch. 2, Ch. 6]{Haglund1}). Using the same weight on flows in $\F_G$ for threshold graphs $G$ does not give symmetric functions with a Schur expansion with coefficients in $\mathbf{N}[q,t]$. 

	\subsection{A note on the statistic pmaj} \label{subsec:pmaj}

	Loehr and Remmel \cite{LR} defined a statistic $\pmaj$ on parking functions and showed that $(\dinv, \area)$ and $(\area, \pmaj)$ are equidistributed. Hence Theorem~\ref{HLhilb} implies that 
	\[\hilb_{q,t}(DH_n) = \sum_P q^{\area(P)} t^{\pmaj(P)}.\]

	One way to define $\pmaj$ is as follows. Parking functions have a natural partial order: $P \leq Q$ if $P(i) \leq Q(i)$ for all $i$. For the complete graph, the maximal parking functions are those with area 0, namely the bijections $Q \colon [n] \to \{0, 1, \dots, n-1\}$. For a maximal parking function $Q$, \[\pmaj(Q) = \sum_{i \colon Q(i) < Q(i+1)} (n-i),\]
	while for any other parking function $P$, $\pmaj(P) = \min_{Q>P} \pmaj(Q)$. Note that on the maximal parking functions,
	\[\sum_{Q \text{ maximal}} t^{\pmaj(Q)} = \hilb_{0,t}(DH_n) = [n]_t!\]
	(and on maximal parking functions, $\pmaj$ is easily seen to be equidistributed with major index on permutations).

	The $\area$ and $\codeg$ statistics coincide on parking functions, which suggests that, in accordance with Conjecture~\ref{conj1}, if there exists a statistic $\stat$ on $G$-parking functions such that
	\[\Ehr_{q,t}(\F_G) = \sum_P q^{\codeg(P)} t^{\stat(P)},\]
	then we should be able to construct $\stat$ to be analogous to $\pmaj$. Then we might expect to be able to define $\stat$ on $G$-parking functions such that
	\[\sum_{Q \text{ maximal}} t^{\stat(Q)} = \Ehr_{0,t}(\F_G) = [\bar d_i]_t!\]
	by Theorem~\ref{thm:t=0}, while for any other parking function $P$, $\stat(P) = \min_{Q>P} \stat(Q)$. This could simplify the task of defining $\stat$ since it would only need to be defined on the maximal $G$-parking functions. Initial computations suggest that it is possible to find such a statistic for small graphs, though we have not yet found a suitable statistic that works for all threshold graphs.


\bibliographystyle{plain}

\end{document}